\documentclass[10pt]{article}
\usepackage{import}
\usepackage{filecontents}
\usepackage{catchfilebetweentags}

\newif\ifLONGVER 
\LONGVERtrue

\newif\ifSTANDALONE
\STANDALONEtrue
\STANDALONEfalse

\usepackage[perpage,symbol]{footmisc}
\usepackage{CJKutf8}
\usepackage[utf8]{inputenc} 
\usepackage{amsthm}
\usepackage{amsfonts}
\usepackage{scalerel}
\usepackage{amssymb}
\usepackage{ mathrsfs }
\usepackage{mathtools}
\usepackage{xcolor}
\usepackage{colonequals}
\usepackage{caption}
\usepackage{environ}
\usepackage{tikz}
\usetikzlibrary{automata,positioning,arrows.meta}
\tikzset{
->>,
>=stealth,
node distance=5cm,
every state/.style={thick, fill=gray!10},
initial text=$\blacksquare$,
every edge/.append style={font=\footnotesize},
initial distance=0.5cm
}
\usepackage[normalem]{ulem}
\usepackage[shortlabels]{enumitem}
\usepackage[citestyle=alphabetic,style=alphabetic,giveninits=true,natbib=true,backend=biber,backref]{biblatex}

\DeclareSourcemap{
  \maps[datatype=bibtex]{
    \map[overwrite=true]{
      \step[fieldset=url, null]
	  \step[fieldset=doi, null]
	  \step[fieldset=isbn, null]
	  \step[fieldset=issn, null]
    }
  }
}
\usepackage[nottoc]{tocbibind}
\usepackage{tocloft}
\usepackage[pdftex,linktocpage]{hyperref}
\usepackage[noabbrev]{cleveref}

\newtheoremstyle{breakplain}
  {}
  {}
  {\itshape}
  {}
  {\bfseries}
  {.}
  {\newline}
  {}
  
\newtheoremstyle{breakdefinition}
  {}
  {}
  {}
  {}
  {\bfseries}
  {.}
  {\newline}
  {}

\theoremstyle{breakdefinition}
\newtheorem{definition}{Definition}[section]

\theoremstyle{remark}
\newtheorem{remark}[definition]{Remark}

\theoremstyle{remark}
\newtheorem{observation}[definition]{Observation}

\theoremstyle{remark}

\theoremstyle{breakplain}
\newtheorem{lemma}[definition]{Lemma}

\theoremstyle{breakdefinition}
\newtheorem{corollary}[definition]{Corollary}

\theoremstyle{breakplain}
\newtheorem{theorem}[definition]{Theorem}

\theoremstyle{breakplain}
\newtheorem*{theorem-no}{Theorem}

\theoremstyle{breakplain}
\newtheorem{proposition}[definition]{Proposition}

\theoremstyle{remark}

\theoremstyle{remark}
\newtheorem{notation}[definition]{Notation}

\theoremstyle{remark}
\newtheorem*{notation-no}{Notation}

\theoremstyle{breakdefinition}

\theoremstyle{breakdefinition}

\theoremstyle{breakdefinition}

\theoremstyle{breakplain}

\theoremstyle{breakplain}

\theoremstyle{breakdefinition}

\newcommand{\lang}[1]{\mscr{L}_{#1}}

\newcommand{\bb}{\mathbb}

\newcommand{\mcal}{\mathcal}
\newcommand{\mscr}{\mathscr}

\newcommand{\mfrak}{\mathfrak}

\newcommand{\pow}{\mcal{P}}

\newcommand{\wo}{\setminus}
\newcommand{\msotimes}{\otimes}
\newcommand{\upperRomannumeral}[1]{\uppercase\expandafter{\romannumeral#1}}

\makeatletter

\newcommand{\Rmn}[1]{\expandafter\@slowromancap\romannumeral #1@}
\makeatother
\makeatletter
\newcommand{\oset}[3][0ex]{%
  \mathrel{\mathop{#3}\limits^{
    \vbox to#1{\kern-2\ex@
    \hbox{$\scriptstyle#2$}\vss}}}}
\makeatother
\newcommand{\exle}{\oset[0ex]{\scriptscriptstyle \exists}{<}}

\newcommand{\efi}{\forall}
\newcommand{\efii}{\exists}

\newcommand{\lift}[1]{\, \hat{#1} \,}

\DeclarePairedDelimiter\dvert{\lvert}{\rvert}

\DeclarePairedDelimiter\dangle{\langle}{\rangle}
\DeclareMathOperator{\Th}{Th}

\DeclareMathOperator{\At}{At}

\DeclareMathOperator{\Fin}{Fin}

\DeclareMathOperator{\Clop}{Clop}

\DeclareMathOperator{\Comp}{Comp}

\DeclareMathOperator{\MO}{MO}

\DeclareMathOperator{\MSO}{M}

\DeclareMathOperator{\up}{up}
\DeclareMathOperator{\LT}{LT}
\DeclareMathOperator{\FSpec}{SPEC}

\DeclareMathOperator{\base}{base}
\DeclareMathOperator{\EF}{EF}
\DeclareMathOperator{\dwcl}{dwcl}

\DeclareMathOperator{\LO}{LO}
\DeclareMathOperator{\Sch}{Sch}

\setlist[enumerate,itemize]{itemsep=1pt,topsep=0pt}

\title{The Pseudofinite Monadic Second Order Theory of Linear Order.}
\author{Deacon Linkhorn}
\begin{document}
\maketitle
\tableofcontents
\clearpage
\section{Introduction.}

This note covers several results about linear orders within the framework of monadic second order logic. In particular the pseudofinite monadic second order theory of linear order. In monadic second order logic we have all the expressive power of first order logic, as well as the ability to quantify over unary relations. To capture monadic second order logic we use a one sorted first order setup which we fully describe in \cref{SecPrelim}, this setup allows for direct use of classical model theoretical techniques and results, such as Ehrenfeucht-Fra\"isse games, and the compactness and completeness theorems for first order logic. 

In \cite{BSCountableOrd} the monadic second order theory of countable ordinals was studied in great depth, building on and consolidating results from \cite{BSAxOmega1}. There, among other results, an axiomatisation of the shared monadic second order theory of countable ordinals is given, as well as a classification of its complete extensions. We investigate what happens when we restrict attention to finite linear orders. Our \cref{TMFinMainTheorem} establishes the correctness of our recursive axiomatisation $T_{\MSO(\Fin)}$ (\cref{DefTMFin}) of the shared monadic second order theory of finite linear orders, i.e. the \emph{pseudofinite monadic second order theory of linear order}. The axiomatisation $T_{\MSO(\Fin)}$ comprises axioms for an atomic Boolean algebra in which all (parametrically) definable sets of atoms are witnessed by elements of the Boolean algebra, having a discrete linear order on the atoms, and such that every non-bottom element of the Boolean algebra contains a smallest and largest atom. In \cref{DefTBasePlus} we give a binary operation $\msotimes$ on a suitable class of structures, which restricts to a well defined binary operation on the set of completions of this shared theory generalising the concatenation of finite linear orders (viewed as monadic second order structures). This operation plays a crucial role in \cref{SecProfinite}.

We then classify the completions of this shared theory in \cref{SecCompletions}, using what we call residue functions (\cref{DefResFunc}). For a finite model the completion is given by appending a sentence saying the Boolean algebra contains a particular number of atoms, for an infinite model we must say what the residue is when we divide the underlying linear order by natural numbers. This is made rigorous by considering minimal subsets of the linear order closed under the $n$-th power of the successor and predecessor functions for each natural number $n$. There are precisely $n$ such sets in each infinite model and they form a partition of the atoms by definable elements. By considering which of these sets contain the smallest and greatest atoms, we get a notion of residue modulo $n$ for infinite models. In \cref{ThmTMFinCompletions} we prove that each residue function (sequence of residues satisfying certain conditions) induces a complete extension of $T_{\MSO(\Fin)}$, and that every completion which admits an infinite model is of this form. 

In \cref{SecProfinite} we explore a connection between the pseudofinite monadic second order theory of linear order and profinite algebra. The model-theoretical analysis of this theory is used to give a new perspective on the free profinite monoid on one generator\footnote{See \cite{AlmeidaProfinBook} subsection 4.4 for an algebraic introduction to this structure and its abstract properties.}, this follows in the footsteps of \cite{vGSteinberg} where model-theoretical analysis of the shared \emph{first order} theory of words\footnote{Given a finite non-empty set $\Sigma$, a $\Sigma$-word is a linear order $\alpha$ together with a function $\alpha \rightarrow \Sigma$, this can be viewed as a first order structure by viewing the preimage of each $\sigma \in \Sigma$ as a unary relation on $\alpha$.} was used to gain new insights on free pro-aperiodic monoids. In both cases the central idea is to view elements of the respective profinite monoids as $0$-types (i.e. completions) of the respective theories. The bridge connecting the free profinite monoid on one generator and the pseudofinite monadic second order theory of linear order is extended Stone duality (see \cref{ExtStoneDuality}), in particular Theorem 6.1 from \cite{GGP} which says that the extended Stone dual of the free profinite monoid on one generator is the Boolean algebra of regular languages over a singleton alphabet. Our \cref{ThmS0isFPM} exploits this to show that the free profinite monoid on one generator is given by the space of completions of the pseudofinite monadic second order theory together with the binary operation $\msotimes$. 

\ \\
\textbf{Acknowledgements.} I would like to thank Sam van Gool for many helpful discussions and comments on various iterations of the note. 

\clearpage
\section{Preliminaries, the base theory \texorpdfstring{$T_{\base}$}{Tbase}.}\label{SecPrelim}

\begin{definition}
Define $\lang{\MO} = \{\subseteq,<\}$ to be the language comprising two binary relation symbols. For each linear order $(\alpha,<)$ we define an $\lang{\MO}$-structure $\MSO(\alpha)$ by taking,
\begin{enumerate}
\item the universe of $\MSO(\alpha)$ to be $\pow(\alpha)$, the powerset of $\alpha$,
\item $\subseteq$ is the usual set-theoretic inclusion relation,
\item $<$ is interpreted as the ordering of $\alpha$, identifying $i \in \alpha$ with $\{i\} \in \pow(\alpha)$. 
\end{enumerate}
We call the $\lang{\MO}$-structure $\MSO(\alpha)$ the \textbf{monadic second order version of $\alpha$}.
\end{definition}

\begin{observation}\label{BaseTheoryObs}
Let $\alpha$ be any linear order, then the following are all easily seen to be true of the $\lang{\MO}$-structure $\MSO(\alpha)$,
\begin{enumerate}
\item $\MSO(\alpha)$ is an atomic Boolean algebra,
\item $<$ is a linear ordering on the atoms of this Boolean algebra,
\item every subset of $\alpha$ is present in $\MSO(\alpha)$, and therefore a fortiori every $\lang{\MO}$-definable subset of $\alpha$ (taken as a set of singleton subsets of $\alpha$) is present in $\MSO(\alpha)$. We refer to this phenomenon as comprehension.
\end{enumerate}
\end{observation}

We will use the above observation to put together a base theory over which we will work, but will only do so after introducing the first order language $\lang{\msotimes}$. 

\begin{notation}\label{MSONotation}
When working with structures $\MSO(\alpha)$ it is natural to distinguish between singletons/atoms and general elements. One way to do this is to work with a two-sorted language and replace $\MSO(\alpha)$ with the two-sorted structure $(\pow(\alpha),\alpha;\in,<)$. We prefer to persevere with a one-sorted language and adopt the convention that uppercase variables are used to quantify over the universe of $\MSO(\alpha)$, while lower case variables are used to quantify over the (uniformly) definable set of atoms. We also use the formula $X(x)$ as shorthand for $x \subseteq X$. 

To demonstrate the utility of this approach, consider the following easily readable sentence,
\[
\forall x \exists Y \forall z (Y(z) \leftrightarrow z<x).
\]
For each linear order $\alpha$ we can easily understand,
\[
\MSO(\alpha) \models \forall x \exists Y \forall z (Y(z) \leftrightarrow z<x),
\]
as saying that the strict downset of any atom is present in $\MSO(\alpha)$. Now instead consider the same formula but written without notational shorthands,
\[
\forall X (\At(X) \rightarrow (\exists Y (\forall Z (\At(Z) \rightarrow (Z \subseteq Y \leftrightarrow Z<X))))),
\]
where $\At(X)$ is the formula,
\[
\exists W (W \subseteq X \wedge W \neq X \wedge \forall U (U \subseteq X \wedge U \neq X \rightarrow U = W)).
\]
The above shorthands make sense whenever we work with atomic Boolean algebras. 
\end{notation}

\begin{remark}
The properties given in \cref{BaseTheoryObs} are easily seen to be expressible using $\lang{\MO}$. Being an atomic Boolean algebra with atoms linearly ordered is expressible using a single $\lang{\MO}$-sentence, for comprehension we must use an axiom schema (see \cref{DefComprehension} for details). Picking up on \cref{MSONotation}, it is worth mentioning that in the two-sorted set up in which the monadic second order version of a first order structure is just a two-sorted first order structure, the comprehension axioms are often subsumed into the semantics. When this is done the resulting semantics is the so-called `Henkin semantics'. For more details see section 9 of \cite{SEPHOL}.
\end{remark}

\begin{definition}
We define $\lang{\msotimes} = \{\subseteq,\bot,\At,\exle\}$ to be the first order language in which $\subseteq$ and $\exle$ are binary relation symbols, $\bot$ is a unary constant symbol, and $\At$ is a unary relation symbol. For each linear order $\alpha$ we define an $\lang{\msotimes}$-structure $\widetilde{\MSO}(\alpha)$ as follows,
\begin{enumerate}
\item the universe is $\pow(\alpha)$,
\item $\subseteq$ is interpreted as set-theoretic inclusion,
\item $\bot$ is interpreted as $\emptyset$,
\item $\At$ is interpreted as the collection of singleton subsets,
\item $\exle$ is interpreted by taking, for each $A,B \in \pow(\alpha)$,
\[
A \exle B \Longleftrightarrow \exists a \in A \exists b \in B (a<b),
\]
where $a<b$ refers to the linear ordering of $\alpha$.
\end{enumerate}
\end{definition}

\begin{lemma}
For each linear order $\alpha$, the structures $\MSO(\alpha)$ and $\widetilde{\MSO}(\alpha)$ are extensions by definitions of one another.
\end{lemma}
\begin{proof}
Let $\alpha$ be a linear order. We have already seen that the collection of atoms of $\MSO(\alpha)$ is defined by the $\lang{\MO}$-formula,
\[
\At(X): \forall W (W \subseteq X \rightarrow (W=X \vee \forall U (U \subseteq W\rightarrow W=U))).
\]
Clearly the bottom element $\bot$ is the unique element satisfying,
\[
\bot(X): \forall Y (X \subseteq Y),
\]
and $\exle$ is defined by the formula,
\[
X \exle Y: \exists W \exists V (W \subseteq X \wedge V \subseteq Y \wedge W<V).
\]

Going in the other direction, $<$ is definable in $\widetilde{\MSO}(\alpha)$ by the formula,
\[
X < Y: \At(X) \wedge \At(Y) \wedge X \exle Y.
\]
So we are done. 
\end{proof}

\begin{remark}
We use $\lang{\msotimes}$ rather than $\lang{\MO}$ for technical reasons, most notably to allow for an easy presentation of \cref{SumEFThm} (if we work with $\lang{\MO}$ rather than $\lang{\msotimes}$ the theorem as stated is false). 
\end{remark}

\begin{definition}\label{DefComprehension}
For each $\lang{\msotimes}$-formula $\eta(x;\bar{Y})$, $\Comp_{\eta}$ is defined to be the following $\lang{\msotimes}$-sentence,
\[
\forall \bar{Y} \exists X (\forall x (X(x) \leftrightarrow \eta(x;\bar{Y}))).
\]
This is a comprehension/separation axiom, it says the set of atoms defined by $\eta$ (over any fixed tuple of parameters $\bar{Y}$) is present as a single element. 
\end{definition}

\begin{definition}
We define $T_{\base}$ to be the following $\lang{\msotimes}$-theory,
\begin{enumerate}
\item atomic Boolean algebra under $\subseteq$,
\item atoms are linearly ordered by $\exle$,
\item $\bot$ is the bottom element of this Boolean algebra (i.e. $\bot$ is the smallest element with respect to $\subseteq$),
\item $\At$ holds precisely for the atoms of the Boolean algebra (i.e. the minimal elements with respect to $\subseteq$ after excluding $\bot$),
\item the restriction of $\exle$ to the atoms determines the whole relation in a way analogous to the standard models, meaning more precisely that,
\[
\forall X \forall Y (X \exle Y \leftrightarrow \exists x \exists y (X(x) \wedge Y(y) \wedge x \exle y)),
\]
\item $\Comp_{\eta}$ for each $\lang{\msotimes}$-formula $\eta(x;\bar{Y})$.
\end{enumerate}
Note that this is just the base theory described in \cref{BaseTheoryObs}, given for $\lang{\msotimes}$ rather than $\lang{\MO}$.
\end{definition}

\begin{remark}
There are $\lang{\msotimes}$-structures which are models of $T_{\base}$ but are not isomorphic to $\MSO(\alpha)$ for any linear order $\alpha$. Let $\alpha$ be any countably infinite linear order and consider $\MSO(\alpha)$, as we have fit everything into a one-sorted first order framework it is crystal clear that we can make use of the L\"owenheim-Skolem theorem. Therefore we are free to take a countable elementary substructure $M \prec \MSO(\alpha)$. Such an elementary substructure comprises a countable Boolean algebra whose atoms form a countably infinite linear order $(\At(M),\exle)$. Therefore $M$ cannot be isomorphic to $\MSO((\At(M),\exle))$ for cardinality reasons.
\end{remark}

\begin{remark}
We want to include $\MSO(\emptyset)$, the monadic second order version of the empty linear order. In \cref{SecProfinite} the theory of $\MSO(\emptyset)$ will serve as the identity element in a monoid. For this reason, when we say Boolean algebra we do not exclude the degenerate $1$ element case in which $\bot = \top$. 
\end{remark}

\begin{remark}
Shelah (\cite{ShelahMTO} Theorem 7 assuming ZFC+CH, CH was later removed) has shown that the monadic second order theory of linear order\footnote{By which is meant $\bigcap_{\alpha \in \LO} \Th(\MSO(\alpha))$ where $\LO$ is the class of linear orders.} is undecidable and therefore $T_{\base}$, which is clearly a recursive set of axioms, cannot serve as an axiomatisation. In particular there must exist $\lang{\msotimes}$-sentences $\phi$ such that $\MSO(\alpha) \models \phi$ for every linear order $\alpha$, while $T_{\base} \not\models \phi$. It is also worth noting that our \cref{SumEFThm} is philosophically downstream from the composition theorems commonly used in model-theoretical analysis of monadic second order theories of linear order, see for example  
\end{remark}

\begin{definition}\label{ResRelDef}
Let $M \models T_{\base}$ and let $A \in M$. Then the structure $M \upharpoonright A$ is the substructure of $M$ with universe the set defined by the formula $X \subseteq A$ in $M$.
\end{definition}

\begin{proposition}\label{ResRelProp}
Let $\phi$ be an $\lang{\msotimes}$-sentence. Then there is an $\lang{\msotimes}$-formula $\phi^{X}$ such that for all $M \models T_{\base}$ and $A \in M$ we have,
\[
M \models \phi^{A} \Longleftrightarrow M \upharpoonright A \models \phi.
\]
\end{proposition}
\begin{proof}
This is a straightforward exercise in relativising quantifiers.
\end{proof}

\begin{remark}\label{ResRelRemark}
We can play around with the above proposition. Here is an example which will be used in \cref{SecProfinite}. Consider the $\lang{\msotimes}$-formula,
\[
\dwcl(X): \forall x \forall y (X(x) \wedge x \exle y \rightarrow X(y)),
\]
in any model of $T_{\base}$ it is easy to see this defines precisely the set of downwards closed elements with respect to the linear ordering of the atoms. Now for any $\lang{\msotimes}$-sentence $\phi$ we can produce a sentence,
\[
\exists X (\dwcl(X) \wedge \phi^X),
\]
and this is true in a model $M$ precisely if there is $A \in M$ such that $A$ is downwards closed and $M \upharpoonright A \models \phi$. 
\end{remark}

\begin{definition}
Recall that if $(\alpha,<_{\alpha})$ and $(\beta,<_{\beta})$ are linear orders, the sum $\alpha + \beta$ is defined to be the linear order obtained by placing $\beta$ `after' $\alpha$ so that each element of $\alpha$ is strictly less than each element of $\beta$.\footnote{Formally this is the linear order with underlying set $\alpha \sqcup \beta$ and with $<_{\alpha + \beta} = <_{\alpha} \cup <_{\beta} \cup (\alpha \times \beta)$.} 
\end{definition}

\begin{remark}
In the following definition we will use the symbol $\msotimes$ to denote an operation on the class of models of $T_{\base}$. 
This operation can be thought of as taking the direct product of Boolean algebras and sum of linear orders at the same time.
\end{remark}

\begin{definition}\label{DefTBasePlus}
Let $M,N \models T_{\base}$. We define an $\lang{\msotimes}$-structure $M \msotimes N$ as follows,
\begin{enumerate}
\item the universe of the structure is $\dvert{M} \times \dvert{N}$,
\item $\subseteq$ is defined as it is in the direct product, i.e. for all $(A,B),(C,D) \in M \msotimes N$ we declare $M \msotimes N \models (A,B) \subseteq (C,D)$ iff $M \models A \subseteq C$ and $N \models B \subseteq D$,
\item $\bot$ and $\At$ are, respectively, the bottom element and set of atoms for the resulting atomic Boolean algebra $(M \msotimes N,\subseteq)$,
\item $\exle$ is given  by $M \msotimes N \models (A,B) \exle (C,D)$ if and only if 
\[
M \not\models A = \bot \text{ and } N \not\models D = \bot, \text{ or }  M \models A \exle C, \text{ or }N \models B \exle D.
\]
\end{enumerate}
\end{definition}

\begin{observation}\label{PlusExtendsStandard}
For any linear orders $\alpha$ and $\beta$ we get $\MSO(\alpha) \msotimes \MSO(\beta) \cong \MSO(\alpha + \beta)$. \Cref{DefTBasePlus} is an extension of this to other models of $T_{\base}$.
\end{observation}

\begin{remark}
For the remainder of \cref{SecPrelim} and continuing into \cref{SecAxioms} we will use Ehrenfeuct-Fra\"isse games. For more details see \cite{HodgesBible}. In particular the definition of unnested formulas on page 58, the definition of quantifier rank on page 103, and also the entirety of sections 3.2 and section 3.3 where two variants of Ehrenfeuct-Fra\"isse games are described and analysed. We write $G_k(M,N)$ where Hodges writes $\EF_k[M,N]$, and will give the definition of this game now.
\end{remark}

\begin{definition}\label{DefEFGame}
Fix a natural number $k$. Let $\lang{}$ be a finite first order language and take two $\lang{}$-structures $M$ and $N$. The \textbf{unnested Ehrenfeucht game of length $k$} is a game played between two players $\efi$ and $\efii$ across $k$ turns. At each turn player $\efi$ first chooses one of the structures $M$ or $N$, and then chooses an element belonging to that structure. Then the player $\efii$ must respond by choosing an element in the remaining structure. Denote the element in $M$ chosen in the $i$-th turn $a_i$, and the element chosen in $N$ in the $i$-th turn $b_i$. After $k$ turns we therefore have two $k$-tuples $\bar{a} = (a_1,\ldots,a_k)$ and $\bar{b} = (b_1,\ldots,b_k)$. Player $\efii$ is declared winner after the play $(\bar{a},\bar{b})$ if and only if for every unnested atomic formula $\phi(x_1,\ldots,x_k)$,
\[
M \models \phi(\bar{a}) \Leftrightarrow N \models \phi(\bar{b}),
\]
and $\efi$ is declared winner if $\efii$ is not declared winner.\footnote{Note in the case $k = 0$ the criterion for declaring that $\efii$ is winner says that $M$ and $N$ agree on unnested atomic $\lang{}$-sentences.}
\end{definition}

\begin{definition}
Let $\lang{}$ be a finite first order language.\footnote{In some presentations of the Ehrenfeuct-Fra\"isse technique, function and constant symbols are prohibited. By restricting to unnested atomic formulas they can be accommodated.} We define $\approx_k$ to be the equivalence relation on the class of $\lang{}$-structures given by the following equivalent\footnote{The equivalence of 1 and 2 is a consequence of Theorem 3.3.2 on page 104 of \cite{HodgesBible}.} definitions,
\begin{enumerate}
\item $M \approx_k N$ if and only if $M$ and $N$ agree on \emph{unnested} $\lang{}$-sentences of quantifier rank at most $k$,
\item $M \approx_k N$ if and only if $\efii$ has a winning strategy in the Ehrenfeucht game $G_k(M,N)$.
\end{enumerate}
\end{definition}

\begin{theorem}\label{SumEFThm}
Let $M_1,M_2,N_1,N_2 \models T_{\base}$. Let $k \in \bb{N}$ and suppose that $M_1 \approx_k N_1$ and $M_2 \approx_k N_2$. Then $M_1 \msotimes M_2 \approx_k N_1 \msotimes N_2$. 
\end{theorem}
\begin{proof}
We argue using the unnested Ehrenfeucht games outlined in \cref{DefEFGame}. Our assumption is that $\efii$ has winning strategies in the games $G_k(M_1,N_1)$ and $G_k(M_2,N_2)$. We will show that $\efii$ has a winning strategy in the game $G_k(M_1 \msotimes M_2,N_1 \msotimes N_2)$, in fact we will explicitly give a winning strategy in terms of those we assumed to exist.

Here is the strategy promised, we outline the response of $\efii$ to a move $(A_i,B_i) \in M_1 \msotimes M_2$ by $\efi$ in the $i$-th turn, the overall strategy is obvious from this description,
\begin{enumerate}
\item upon receiving the $i$-th move $(A_i,B_i) \in M_1 \msotimes M_2$ from $\efi$, take both $A_i$ and $B_i$ and feed them into simulations of the games $G_k(M_1,N_1)$ and $G_k(M_2,N_2)$ respectively,
\item take the responses $C_i \in N_1$ and $D_i \in N_2$ in those simulated games given by the winning strategies we assumed to exist,
\item form the move $(C_i,D_i) \in N_1 \msotimes N_2$ in the main game as a response.
\end{enumerate}

In order to verify that this is a winning strategy, suppose that a play has been completed. Unnested atomic formulas in $X_1,\ldots,X_k$ are of the form $X_i = X_j$, $X_i \subseteq X_j$, $\At(X_i)$, or $X_i \exle X_j$ where $1 \leq i,j \leq k$.\footnote{Here one or both of the variables $X_i$ and $X_j$ could be replaced with $\bot$, we leave it to the reader to think through these cases.} We will deal with the case $X_i \exle X_j$, and leave the remaining cases as an exercise. Suppose,
\[
M_1 \msotimes M_2 \models (A_i,B_i) \exle (A_j,B_j).
\] 
Then by definition of $\msotimes$ (\cref{DefTBasePlus}) at least one of the following must hold,
\begin{enumerate}
\item $M_1 \not\models A_i = \bot$ and $M_2 \not\models B_j = \bot$,
\item $M_1 \models A_i \exle A_j$, 
\item $M_2 \models B_i \exle B_j$.
\end{enumerate}
Now by our assumption that the pairs $(\bar{A},\bar{C})$ and $(\bar{B},\bar{D})$ are winning plays for $\efii$ in $G_k(M_1,N_1)$ and $G_k(M_2,N_2)$ respectively, this implies that at least one of the following must hold,
\begin{enumerate}
\item $N_1 \not\models C_i = \bot$ and $N_2 \not\models D_j = \bot$,
\item $N_1 \models C_i \exle C_j$, 
\item $N_2 \models D_i \exle D_j$.
\end{enumerate}
This then implies, again by definition of $\msotimes$, that,
\[
N_1 \msotimes N_2 \models (C_i,D_i) \exle (C_j,D_j).
\]
By symmetry we can therefore conclude,
\[
M_1 \msotimes M_2 \models (A_i,B_i) \exle (A_j,B_j) \Leftrightarrow N_1 \msotimes N_2 \models (C_i,D_i) \exle (C_j,D_j),
\]
precisely as required. 
\end{proof}

\begin{corollary}\label{SumEFCor}
Let $M_1,M_2,N_1,N_2 \models T_{\base}$. Suppose that $M_1 \equiv N_1$ and $M_2 \equiv N_2$. Then $M_1 \msotimes M_2 \equiv N_1 \msotimes N_2$. 
\end{corollary}
\begin{proof}
From \cref{SumEFThm} using the fact that elementary equivalence is the intersection over $\bb{N}$ of the relations $\approx_k$.
\end{proof}

\clearpage
\section{The axiomatisation \texorpdfstring{$T_{\MSO(\Fin)}$}{TMFin}.}\label{SecAxioms}

\begin{definition}\label{DefTMFin}
The $\lang{\msotimes}$-theory $T_{\MSO(\Fin)}$ is defined to be $T_{\base}$ extended by axioms which say that,
\begin{enumerate}
\item the linear ordering $\exle$ of the atoms is discrete with endpoints,
\item every non-empty element contains a smallest atom with respect to the ordering $\exle$, that is,
\[
\forall X (X \neq \bot \rightarrow (\exists x (X(x) \wedge \forall y (y \exle x \rightarrow \neg X(y))))).
\]
\end{enumerate}
\end{definition}

\begin{definition}
For each $n \in \bb{N}$, we write $\MSO(n)$ for the monadic second order version of the unique finite linear order with $n$ elements. By the \textbf{pseudofinite monadic second order theory of linear order} we mean the $\lang{\MO}$-theory $\bigcap_{n \in \bb{N}} \Th(\MSO(n))$.
\end{definition}

\begin{notation}
In each model of $T_{\MSO(\Fin)}$ there are smallest and largest atoms which are clearly definable elements. We will write $0$ as shorthand for the smallest atom, and $0^*$ as shorthand for the largest atom.\footnote{We choose $0^*$ because for any infinite $M \models T_{\MSO(\Fin)}$, $\omega + \omega^*$ embeds (in precisely one way) as a linear order \emph{with successor and predecessor functions} into $(\At(M),<,s,p)$. Our preferred presentation of $\omega + \omega^*$ is $\{0,1,2,\ldots\} \cup \{\ldots,2^*,1^*,0^*\}$.} 
\end{notation}

\begin{remark}
The axioms from $T_{\MSO(\Fin)}$ imply that every non-empty element contains a largest atom with respect to $\exle$. To see this note that if $M \models T_{\MSO(\Fin)}$ and $A \in M$, then by comprehension there is an element $B \in M$ such that $B$ contains precisely those atoms which are strictly greater than every atom contained in $A$. Then either $B$ is empty, in which case the right endpoint of the linear ordering of the atoms is the largest atom in $A$, or $B$ is non-empty, in which case $B$ contains a smallest atom and the predecessor of this atom is easily seen to be the largest atom in $A$. 
\end{remark}

The following definition is just a special case of \cref{ResRelDef} which makes sense for models of $T_{\MSO(\Fin)}$ but not for arbitrary models of $T_{\base}$.

\begin{definition}
Let $M \models T_{\MSO(\Fin)}$ and $a \in \At(M)$. Then we write $M \upharpoonright [0,a)$ for the restriction of $M$ to the element defined by the formula (in free variable X),
\[
\forall y (X(y) \leftrightarrow y<a).
\]
\end{definition}

\begin{proposition}\label{bddformulalemma}
Let $\phi$ be an $\lang{\msotimes}$-sentence. Then there is an $\lang{\msotimes}$-formula $\phi^{<x}$ such that for all $M \models T_{\MSO(\Fin)}$ and $a \in \At(M)$ we have,
\[
M \models \phi^{<a} \Longleftrightarrow M \upharpoonright [0,a) \models \phi.
\]
\end{proposition}
\begin{proof}
This is an easy consequence of \cref{ResRelProp} and the definition of $T_{\MSO(\Fin)}$.
\end{proof}

\begin{lemma}\label{ResIndLemma}
Let $M \models T_{\MSO(\Fin)}$ and let $\phi(x)$ be a formula in one free variable. Let $s(x)=y$ be shorthand for the definable relation `$x\exle 0^*$ and $y$ is the immediate successor of $x$'. Then,
\[
M \models (\phi(0) \wedge \forall x \forall y((\phi(x)  \wedge s(x)=y) \rightarrow \phi(y))) \rightarrow \forall x \phi(x).
\]
\end{lemma}
\begin{proof}
This is a straightforward exercise in using the axioms of $T_{\MSO(\Fin)}$. By comprehension there is an element $A \in M$ such that $M \models \forall x (\neg \phi(x) \leftrightarrow A(x))$. This element $A$, if non-empty, contains a smallest atom. If this smallest atom is $0$ we are done, else we can consider the immediate predecessor of the smallest atom (by discreteness of the linear order) and then we are done. 
\end{proof}

\begin{remark}
The above is referred to as `restricted induction' by Doets (see \cite{Doets} page 86), restricted presumably referring to the fact that the linear order admits a greatest element.
\end{remark}

\begin{proposition}\label{MSOAddOneProp}
Let $M \models T_{\MSO(\Fin)}$ and suppose that for some $k,n \in \bb{N}$ we have $M \approx_k \MSO(n)$. Then $M \msotimes \MSO(1) \approx_k \MSO(n+1)$.
\end{proposition}
\begin{proof}
Trivially we have $\MSO(1) \approx_k \MSO(1)$. Therefore if we assume that $M \approx_k \MSO(n)$ \cref{SumEFThm} gives us that $M \msotimes \MSO(1) \approx_k \MSO(n) \msotimes \MSO(1)$. But then as we noted in \cref{PlusExtendsStandard}, $\MSO(n) \msotimes \MSO(1) \cong \MSO(n+1)$ so we are done. 
\end{proof}

The proof of the following theorem follows the same outline as the proof of \cite{vGSteinberg} Proposition 4.2. 

\begin{proposition}\label{TMFinEFStandardThm}
For each $k \in \bb{N}$ and $M \models T_{\MSO(\Fin)}$, there is $n \in \bb{N}$ such that $M \approx_k \MSO(n)$. 
\end{proposition}
\begin{proof}
To start, fix $k \in \bb{N}$. By the Fra\"isse-Hintikka theorem (\cite{HodgesBible} Theorem 3.3.2) there is an $\lang{\msotimes}$-sentence $\phi_k$ such that for each $\lang{\msotimes}$-structure $M$, $M \models \phi_k$ iff there is $n \in \bb{N}$ such that $M \approx_k \MSO(n)$. This is because there are finitely many equivalence classes of the relation $\approx_k$, and each of these equivalence classes is axiomatised by a single $\lang{\msotimes}$-sentence. Taking just those equivalence classes which contain $\MSO(n)$ for some $n \in \bb{N}$, and then forming the finite disjunction of the sentences axiomatising each of the equivalence classes we get the required sentence $\phi_k$.

Next take $M \models T_{\MSO(\Fin)}$ and consider the formula $\phi_k^{<x}$ given by \cref{bddformulalemma}. Recall that by construction $\phi_k^{<x}$ is such that for each $a \in \At(M)$,
\[
M \models \phi_k^{<a} \Leftrightarrow M \upharpoonright [0,a) \models \phi_k.
\]
From $M \models T_{\MSO(\Fin)}$, $M \upharpoonright [0,0) \cong \MSO(0)$, and therefore $M \models \phi_k^{<0}$. Now by \cref{MSOAddOneProp} we get for any $a,b \in \At(M)$ that ,
\[
M \models (s(a)=b \rightarrow (\phi_k^{<a} \rightarrow \phi_k^{<b})).
\]
By \cref{ResIndLemma} we therefore get $M \models \phi_k^{<0^*}$, so $M \upharpoonright [0,0^*) \approx_k \MSO(n)$ for some $n \in \bb{N}$. Since $M \cong M \upharpoonright [0,0^*) \msotimes \MSO(1)$ \cref{MSOAddOneProp} gives $M \approx_k \MSO(n+1)$ and hence $M \models \phi_k$, so we are done. 
\end{proof}

\begin{theorem}\label{TMFinMainTheorem}
$T_{\MSO(\Fin)}$ is an axiom system for the pseudofinite monadic second order theory of linear order. 
\end{theorem}
\begin{proof}
It is clear that for each $n \in \bb{N}$ we have $\MSO(n) \models T_{\MSO(\Fin)}$. It is therefore enough to show that for each $M \models T_{\MSO(\Fin)}$ and $\lang{\msotimes}$-sentence $\phi$, if $M \models \phi$ then there exists $n \in \bb{N}$ such that $\MSO(n) \models \phi$. This is immediate from \cref{TMFinEFStandardThm}, taking $k$ to be the quantifier rank of $\phi$ (more accurately, the quantifier rank of some unnested $\lang{\msotimes}$-sentence which is logically equivalent to $\phi$).
\end{proof}

\begin{proposition}\label{TMFinSumThm}
Let $M,N \models T_{\MSO(\Fin)}$, then $M \msotimes N \models T_{\MSO(\Fin)}$.
\end{proposition}
\begin{proof}
By \cref{TMFinMainTheorem} it is enough to show, for each $M,N \models T_{\MSO(\Fin)}$ and $k \in \bb{N}$, that there exists $n \in \bb{N}$ such that $M \msotimes N \approx_k \MSO(n)$. 

Fix $k \in \bb{N}$, then there exists $n_1 \in \bb{N}$ such that $M \approx_k \MSO(n_1)$, and there exists $n_2 \in \bb{N}$ such that $N \approx_k \MSO(n_2)$. But then we are done by \cref{PlusExtendsStandard,SumEFThm} as taking $n=n_1 + n_2$ we get $M \msotimes N \approx_k \MSO(n)$ as required.
\end{proof}

\begin{corollary}\label{TMFinSumCor}
Let $T,T' \models T_{\MSO(\Fin)}$ be complete and let $M \models T$ and $N \models T'$ be any models. Then taking $T \msotimes T' \coloneqq \Th(M \msotimes N)$ gives a well defined binary operation on the set of complete theories extending $T_{\MSO(\Fin)}$. 
\end{corollary}
\begin{proof}
By \cref{SumEFCor} this gives us a well defined operation, and by \cref{TMFinSumThm} we have $T \msotimes T' \models T_{\MSO(\Fin)}$.
\end{proof}

\clearpage
\section{The completions of \texorpdfstring{$T_{\MSO(\Fin)}$}{TMFin}.}\label{SecCompletions}

In this section we will give a full description of the completions of $T_{\MSO(\Fin)}$. Every finite model of $T_{\MSO(\Fin)}$ is of the form $\MSO(n)$ for some $n \in \bb{N}$, but it is an immediate consequence of compactness that $T_{\MSO(\Fin)}$ has infinite models. Therefore there are completions of $T_{\MSO(\Fin)}$ which are not of the form $\Th(\MSO(n))$ for any $n \in \bb{N}$. 

\begin{notation}
We write $\bb{P}$ for the set of prime natural numbers.
\end{notation}

\begin{CJK*}{UTF8}{gbsn}
\begin{theorem}[Chinese remainder theorem, (秦九韶, 1247)\footnote{See \cite{Dauben} pp 309 onwards for more historical background.}]\label{CRT}
Let $n_1,\ldots,n_k \in \bb{N}$ be pairwise coprime natural numbers. Let $a_1,\ldots,a_k \in \bb{Z}$. Then there exists a unique $x \in \bb{N}$ such that $0 \leq x < \prod_{i=1}^k n_i$ and such that $x \equiv a_i \mod n_i$ for $0 \leq i < k$.
\end{theorem}
\end{CJK*}

\begin{definition}
Let $M \models T_{\MSO(\Fin)}$ be an $\lang{\msotimes}$-structure. Let $d \in \bb{N}_{>0}$. Then for each $0 < h \leq d$ we define $\rho_{d,h}$ to be a sentence saying there exists a partition of $\At(M)$ by $d$ non-bottom elements $A_1,\ldots,A_d \in M$, such that,
\begin{enumerate}
\item $M \models A_1(0)$,
\item for all $a \in \At(M)$ and $0 < h \leq d$, if $M \models a \neq 0^* \wedge A_h(a)$ then $M \models A_{h'}(a')$ where $a'$ is the successor of $a$ in the linear ordering of the atoms of $M$, and $h'$ is the successor of $h$ in the cyclic ordering of $\{1,\ldots,d\}$,
\item $M \models A_h(0^*)$.
\end{enumerate}
\end{definition}

\begin{remark}
Let $n,d \in \bb{N}_{>0}$ be such that $n \geq d$. Then for each $0 < h \leq d$ we have $\MSO(n) \models \rho_{d,h}$ if and only if $n \equiv h \mod d$.
\end{remark}

\begin{definition}
For each $n \in \bb{N}$ we write $\psi_{>n}$ for the $\lang{\msotimes}$-sentence saying that there exist at least $n+1$ distinct atoms, and $\psi_{=n}$ for the sentence saying that there exist precisely $n$ distinct atoms.
\end{definition}

\begin{lemma}\label{UniqueResLemma}
Let $M \models T_{\MSO(\Fin)}$ be infinite. Then for each $d \in \bb{N}_{>0}$, there is precisely one $0 < h \leq d$ such that $M \models \rho_{d,h}$. 
\end{lemma}
\begin{proof}
Let $d \in \bb{N}_{>0}$, by \cref{TMFinMainTheorem},
\[
T_{\MSO(\Fin)} \models \psi_{>d} \rightarrow \bigvee_{0 < h \leq d} (\rho_{d,h} \wedge \bigwedge_{k \neq h} \neg \rho_{d,k}),
\]
holds if and only if,
\[
\MSO(n) \models \bigvee_{0 < h \leq d} (\rho_{d,h} \wedge \bigwedge_{k \neq h} \neg \rho_{d,k}),
\]
for each $n > d$. The latter is obviously true. Now if $M$ is infinite then $M \models \psi_{>d}$ for each $d \in \bb{N}_{>1}$, so we are done. 
\end{proof}

\begin{definition}
Let $M \models T_{\MSO(\Fin)}$ be infinite. Then for $d \in \bb{N}_{>0}$ we define the residue of $M$ modulo $d$ to be the unique $0 < h \leq d$ such that $M \models \rho_{d,h}$.
\end{definition}

\begin{definition}\label{DefResFunc}
A \textbf{residue function} is a function $r:\bb{P} \times \bb{N}_{>0} \rightarrow \bb{N}$ such that for $(p,j) \in \bb{P} \times \bb{N}_{>0}$,
\begin{enumerate}
\item $r(p,j)<p^j$ and,
\item $r(p,j+1) \equiv r(p,j) \mod p^j$.
\end{enumerate}
\end{definition}

\begin{proposition}\label{ResUniqueExtProp}
Let $r: \bb{P} \times \bb{N}_{>0} \rightarrow \bb{N}$ be a residue function. Then there is a unique function $\tilde{r}: \bb{N}_{>1} \rightarrow \bb{N}$ such that for each $d \in \bb{N}_{>1}$,
\begin{enumerate}
\item if $p \in \bb{P}$ and $j \in \bb{N}_{>0}$ are such that $p^j \mid d$, then $\tilde{r}(d) \equiv r(p,j) \mod p^j$,
\item $\tilde{r}(d) < d$.
\end{enumerate}
\end{proposition}
\begin{proof}
First we can establish uniqueness, suppose for contradiction that $q: \bb{N}_{>1} \rightarrow \bb{N}$ and $s: \bb{N}_{>1} \rightarrow \bb{N}$ satisfy the above properties, but that $q(d) \neq s(d)$ for some $d > 1$. Then there is a prime power $p^j$ such that,
\begin{enumerate}
\item $p^j \mid d$,
\item $q(d) \not\equiv s(d) \mod p^j$,
\end{enumerate}  
by \cref{CRT}. But on the other hand $q(d) \equiv r(p,j) \mod p^j$ and $r(p,j) \equiv s(d) \mod p^j$ so by transitivity we have $q(d) \equiv s(d) \mod p^j$, hence a contradiction. This establishes uniqueness.

For existence, we use \cref{CRT}. For each $n \in \bb{N}_{>1}$ we need to be able to choose $\tilde{r}(d) \in \{0,\ldots,d-1\}$ such that for each prime power $p^j$ with $p^j \mid d$ we have $\tilde{r}(d) \equiv r(p,j) \mod p^j$. Take all prime powers $p^j$ such that $p^j \mid n$ and $p^{j+1} \nmid d$, these are clearly coprime so \cref{CRT} tells us there is an element $\tilde{r}(d) < d$ such that $\tilde{r}(d) \equiv r(p,j) \mod p^j$ for each. Then if $p^i$ is a prime power with $p^i \mid d$, we get $\widetilde{r}(d) \equiv r(p,j) \equiv r(p,i) \mod p^i$ where $j$ is maximal such that $p^j \mid d$, the second congruence comes from our assumption that $r$ is a residue function.
\end{proof}

\begin{notation}
We will abuse notation and call a function $\tilde{r}:\bb{N}_{>1} \rightarrow \bb{N}$ a residue function if there exists some residue function $r:\bb{P} \times \bb{N}_{>0} \rightarrow \bb{N}$ such that $r$ and $\tilde{r}$ satisfy the conditions in \cref{ResUniqueExtProp}.
\end{notation}

\begin{lemma}\label{ModelGivesResLemma}
Let $M \models T_{\MSO(\Fin)}$ be infinite and let $\tilde{r}_M:\bb{N}_{>1} \rightarrow \bb{N}$ be the function given by taking $\tilde{r}_M(d)$ to be the unique $0 < h \leq d$ such that $M \models \rho_{d,h}$ for each $d \in \bb{N}_{>1}$. Then $\tilde{r}_M$ is a residue function. 
\end{lemma}
\begin{proof}
We need to check that $\tilde{r}_M(p^{j+1}) \equiv \tilde{r}_M(p^j) \mod p^j$ for each $p \in \bb{P}$ and $j \in \bb{N}_{>0}$. Fix a particular $p \in \bb{P}$ and $j \in \bb{N}_{>0}$ and let $S_{p,j}$ be the finite set $\{i < p^{j+1}: i \equiv r_M(p^j) \mod p^j\}$. We can now reformulate  the condition,
\[
\tilde{r}_M(p^{j+1}) \equiv \tilde{r}_M(p^j) \mod p^j,
\]
as,
\[
M \models \bigvee_{i \in S_{p,j}} \rho_{p^{j+1},i}. 
\]
By considering what happens in sufficiently large standard models and using \cref{TMFinMainTheorem} we get that for each $p \in \bb{P}$ and $j \in \bb{N}_{>0}$,
\[
T_{\MSO(\Fin)} \models \psi_{>p^{j+1}} \rightarrow (\rho_{p^j,i} \rightarrow \bigvee_{i \in S_{p,j}} \rho_{p^{j+1},i}).
\]
Therefore we are done.
\end{proof}

\begin{definition}
Let $\tilde{r}:\bb{N}_{>1} \rightarrow \bb{N}$ be a residue function. Then we define an $\lang{\msotimes}$-theory,
\[
T_{\tilde{r}} \coloneqq T_{\MSO(\Fin)} \cup T_{\infty} \cup \{\rho_{d,\tilde{r}(d)}:d \in \bb{N}_{>1}\}.
\]
Here $T_{\infty}$ comprises, for each $n \in \bb{N}$, the sentence $\psi_{>n}$.
\end{definition}

\begin{proposition}\label{ResConsistentProp}
For each residue function $\tilde{r}$, the $\lang{\msotimes}$-theory $T_{\tilde{r}}$ is consistent. 
\end{proposition} 
\begin{proof}
That each finite subset of $T_{\tilde{r}}$ is consistent is straightforward from \cref{CRT}. 
\end{proof}

\begin{definition}\label{FSpecDef}
Let $\lang{}$ be a first order language, $T$ an $\lang{}$-theory and $\phi$ an $\lang{}$-sentence. Then $\FSpec_T(\phi) \subseteq \bb{N}$ is typically\footnote{See \cite{HodgesBible} page 542 for example, where the notation $\Sch(\phi)$ is used.} defined to be the set of natural numbers,
\[
\FSpec_T{\phi} = \{n \in \bb{N}: \text{ there is } M \models T \cup \{\phi\} \text{ such that }\dvert{M} = n\}.
\]
We will modify this definition as follows, for us,
\[
\FSpec_T{\phi} = \{n \in \bb{N}: \text{ there is } M \models T \cup \{\phi\} \text{ such that }\dvert{M} = 2^n\}.
\]
This modification, when working over the theory of Boolean algebras, essentially means we choose to measure a (finite) structure by looking at the number of atoms rather than using cardinality. Usually we will omit $T$ from the notation when it is clear from the context.
\end{definition}

\begin{definition}\label{DefUltPer}
Recall that a subset $S \subseteq \bb{N}$ is \textbf{ultimately periodic} if there exist $N \in \bb{N}$ and $d \in \bb{N}_{>0}$ such that for all $n \in \bb{N}_{>N}$, $n \in S$ if and only if $n+d \in S$.
\end{definition}

\begin{proposition}\label{SPECUPThm}
For each $\lang{\msotimes}$-sentence $\phi$, the set $\FSpec_{T_{\MSO(\Fin)}}(\phi)$ is ultimately periodic.
\end{proposition}
\begin{proof}
This is a well known consequence of the correspondence between finite automata and monadic second order versions of linear orders (see for example chapter 2 section 10 of \cite{KNAutomataBook}, in particular theorem 2.10.1 on page 113 and theorem 2.10.3 on page 120). For each $\lang{\msotimes}$-sentence $\phi$ a deterministic finite automata over a singleton alphabet can be constructed, such that $\FSpec_{T_{\MSO(\Fin)}}(\phi)$ is precisely the set of lengths of words accepted by the automaton. It is well known and straightforward to show that such sets are ultimately periodic, but to go further would require defining what deterministic finite automata are and so we will not. 
\end{proof}

\begin{theorem}\label{ThmTMFinSentences}
For every $\lang{\msotimes}$-sentence $\phi$ there exist $N \in \bb{N}$, $d \in \bb{N}_{>0}$, $\mfrak{i} \subseteq \{0,\ldots,N\}$ and $\mfrak{r} \subseteq \{1,\ldots,d\}$ such that,
\[
T_{\MSO(\Fin)} \models \phi \leftrightarrow (\bigvee_{i \in \mfrak{i}} \psi_{=i} \vee (\psi_{>N} \wedge \bigvee_{h \in \mfrak{r}} \rho_{d,h})).
\]
\end{theorem}
\begin{proof}
Let $\phi$ be an $\lang{\msotimes}$-sentence. By \cref{SPECUPThm} we have that $\FSpec(\phi)$ is ultimately periodic, i.e. there exists $N \in \bb{N}$ and $d \in \bb{N}_{>0}$ such that for all $n > N$, $n \in \FSpec(\phi)$ if and only if $n+d \in \FSpec(\phi)$. From this periodicity of $\FSpec(\phi)$, there is a finite set $\mfrak{r} \subseteq \{1,\ldots,d\}$ such that for each $n>N$, 
\[
\MSO(n) \models \phi \leftrightarrow \bigvee_{h \in \mfrak{r}} \rho_{d,h}.
\]
Let $\mfrak{i} = \FSpec(\phi) \cap \{0,\ldots,N\}$, so that for $n \leq N$ we have,
\[
\MSO(n) \models \phi \leftrightarrow \bigvee_{i \in \mfrak{i}} \psi_{=i}.
\]
Then using \cref{TMFinMainTheorem} we can combine these two observations to get,
\[
T_{\MSO(\Fin)} \models \phi \leftrightarrow (\bigvee_{i \in \mfrak{i}} \psi_{=i} \vee (\psi_{>N} \wedge \bigvee_{h \in \mfrak{r}} \rho_{d,h})),
\]
as required.
\end{proof}

\begin{theorem}\label{ThmTMFinCompletions}
For each residue sequence $\tilde{r}:\bb{N}_{>1} \rightarrow \bb{N}$, the theory $T_{\tilde{r}}$ is a completion of $T_{\MSO(\Fin)}$, moreover every non-standard completion is of this form. 
\end{theorem}
\begin{proof}
We must show that $T_{\tilde{r}}$ is complete for each residue function $\tilde{r}$. The rest of the statement, namely that every non-standard completion is of the form $T_{\tilde{r}}$ for some residue function $\tilde{r}$, then follows easily from \cref{ModelGivesResLemma}.

So let $\tilde{r}:\bb{N}_{>1} \rightarrow \bb{N}$ be a residue function. Let $\phi$ be an $\lang{\msotimes}$-sentence. By \cref{ThmTMFinSentences} there exist $N \in \bb{N}$, $d \in \bb{N}_{>0}$, $\mfrak{i} \subseteq \{0,\ldots,N\}$ and $\mfrak{r} \subseteq \{1,\ldots,d\}$ such that,
\[
T_{\MSO(\Fin)} \models \phi \leftrightarrow (\bigvee_{i \in \mfrak{i}} \psi_{=i} \vee (\psi_{>N} \wedge \bigvee_{h \in \mfrak{r}} \rho_{d,h})).
\]
But now recall that, by definition of $T_{\tilde{r}}$, $T_{\infty} \subseteq T_{\tilde{r}}$. Therefore,
\[
T_{\tilde{r}} \models \phi \leftrightarrow \bigvee_{h \in \mfrak{r}} \rho_{d,h}.
\]
Now recall that by the definition of $T_{\tilde{r}}$, $\tilde{r}(d)$ is the unique $h \in \{1,\ldots,d\}$ such that $\rho_{d,h} \in T_{\tilde{r}}$. Therefore $T_{\tilde{r}} \models \phi$ if and only if $\tilde{r}(d) \in \mfrak{r}$. This shows that $T_{\tilde{r}}$ is complete.
\end{proof}

\begin{remark}\label{RmkResidueAddition}
Let $\widetilde{\rho}_1,\widetilde{\rho}_2:\bb{N}_{>1} \rightarrow \bb{N}$ be residue functions. Define $\widetilde{\rho}_1 \msotimes \widetilde{\rho}_2:\bb{N}_{>1} \rightarrow \bb{N}$ by taking $(\widetilde{\rho}_1 \msotimes \widetilde{\rho}_2)(d)$ to be the unique element of $\{0,\ldots,d-1\}$ congruent to $\widetilde{\rho}_1(d) + \widetilde{\rho}_2(d)$ modulo $d$, for each $d \in \bb{N}_{>1}$. Then $\widetilde{\rho}_1 \msotimes \widetilde{\rho}_2$ is a residue function, and moreover we get,
\[
T_{\widetilde{\rho}_1} \msotimes T_{\widetilde{\rho}_2} = T_{\widetilde{\rho}_1 \msotimes \widetilde{\rho}_2}.
\]
\end{remark}

\clearpage
\section{Extended Stone duality and the free profinite monoid on one generator.}\label{SecProfinite}

In this section we will show that $(S_0(T_{\MSO(\Fin)}),\msotimes)$ the space of completions of the theory $T_{\MSO(\Fin)}$ together with the operation $\msotimes$ (see \cref{TMFinSumCor}), \emph{is}\footnote{To be fully rigorous, by `is' we mean that these two objects are isomorphic in the category of Boolean spaces equipped with compatible ternary relations described in \cref{ExtStoneDuality}.} the free profinite monoid on one generator. For an algebraic introduction to the free profinite monoid on one generator see \cite{AlmeidaProfinBook}, in particular the beginning of subsection 4.4 on page 96, with the case where $A$ is a singleton in mind.  

\begin{definition}
Let $L$ be a Boolean algebra, and let $\cdot:L^2 \rightarrow L$ be a binary operation such that,
\begin{enumerate}
\item for each $a \in L$, $a \cdot \bot = \bot \cdot a = \bot$,
\item for each $a,b,c \in L$, $(a \vee b) \cdot c = (a \cdot c) \vee (b \cdot c)$ and $a \cdot (b \vee c) = (a \cdot b) \vee (a \cdot c)$.
\end{enumerate} 
Then we say that $\cdot$ is a \textbf{normal additive binary operation}.
\ \\
Let $X$ be a Boolean space, and let $R \subseteq X^3$ be a ternary relation such that,
\begin{enumerate}
\item for each $A,B \in \Clop(X)$,
\[ 
A \cdot B \coloneqq \{z \in X:\exists x \in A,\exists y \in B \; R(x,y,z)\} \in \Clop(X),
\]
\item for each $z \in X$ the set $\{(x,y) \in X^2: R(x,y,z)\} \in \mcal{A}(X^2)$\footnote{Given a topological space $X$, $\mcal{A}(X)$ denotes the collection of closed subsets of $X$. Here $X^2$ is equipped with the product topology.}.
\end{enumerate}
Then we say that $R$ is a \textbf{compatible} ternary relation.
\end{definition}

\begin{theorem}[Extended Stone duality for Boolean algebras with normal additive binary operations.]\label{ExtStoneDuality}
The following categories are dual to each other, 
\begin{enumerate}
\item the category of Boolean algebras equipped with a normal additive binary operation,
\item the category of Boolean spaces equipped with a compatible ternary relation.
\end{enumerate}
Morphisms in the first category are $\lang{}=\{\vee,\wedge,\bot,\top,^c,\cdot\}$-homomorphisms. Morphisms in the second category are more complicated, if $(X,R)$ and $(Y,S)$ are objects in the second category then a function $f:X \rightarrow Y$ is a morphism\footnote{\cite{Goldblatt} refers to these maps as `bounded morphisms'.} if,
\begin{enumerate}
\item $f$ is continuous,
\item for all $x,y,z \in X$, $R(x,y,z)$ implies $S(f(x),f(y),f(z))$, and,
\item for any $z \in X$, if $x',y' \in Y$ with $S(x',y',f(z))$ then there exists $x,y \in X$ such that $f(x)=x'$, $f(y)=y'$, and $R(x,y,z)$.\footnote{See \cite{Hansoul} definition 1.10 and comments following the definition on page 41, this definition is only suitable for the Boolean case, as Boolean spaces are Hausdorff, for spectral spaces we must change the definition to account for the topology, and for Priestley spaces the definition must be changed to account for the ordering.}
\end{enumerate}
This duality is witnessed by the following functors.
\begin{enumerate}
\item Let $(L,\cdot)$ be a Boolean algebra equipped with a finite join preserving operation. Send this to the Boolean space $X$ comprising prime filters of $L$ together with the Zariski topology, equipped with $R_{\bullet} \subseteq X^3$ given by,
\[
R_{\bullet}(x,y,z) \Leftrightarrow x \lift{\cdot} y \subseteq z,
\]
where $x \lift{\cdot} y \coloneqq \{a \cdot b: a \in x,b \in y\}$.
\item Let $(X,R)$ be a Boolean space equipped with a compatible ternary relation. Send this to the Boolean algebra $\Clop(X)$ comprising clopen subsets of $X$, equipped with $\cdot:\Clop(X)^2 \rightarrow \Clop(X)$ given by,
\[
A \cdot B \coloneqq \{z \in X:\exists x \in A,\exists y \in B \; R(x,y,z)\}.
\]
\end{enumerate}
Both categories are concrete, and morphisms are sent to the corresponding preimage maps. When we have either a Boolean algebra with a finite join preserving operation or a Boolean space with a compatible ternary relation, we call its image under the appropriate functor the `\textbf{extended Stone dual}'.
\end{theorem}
\begin{proof}
We do not prove this here. Extending Stone duality to account for operators on a Boolean algebra was first done in \cite{JT1,JT2}. I have formed the above theorem by specialising the results of \cite{Hansoul}, which generalises the results of Jonsson and Tarski to distributive lattices with operators and spectral spaces with relational structure. In \cite{Goldblatt} and \cite{GehrkeJonsson} equivalent generalisations are obtained in terms of Priestley spaces with relational structure. In both references things are done in a fuller generality taking into account higher arity operators, other aspects of the duality are also considered. \cite{Dekkers} is a helpful reference on extended Stone duality for Boolean algebras which is more accessible in places. 
\end{proof}

\begin{definition}
If $\cdot:S^2 \rightarrow S$ is a binary operation on a set $S$, then for any $A,B \in \pow(S)$ we define $A \lift{\cdot} B \coloneqq \{a \cdot b:a \in A,b \in B\}$.
\end{definition}

\begin{definition}\label{DefFunctional}
Let $(L,\cdot)$ be a Boolean algebra with a normal additive binary operation. The extended Stone dual $(X,R_{\bullet})$ of $(L,\cdot)$ is called \textbf{functional} if for each $x,y \in X$ there is a unique $z \in X$ such that $x \lift{\cdot} y \subseteq z$. In this case we write $\times$ for the binary operation $\times:X^2 \rightarrow X$ which takes $(x,y) \in X^2$ to the unique $z \in X$ such that $x \lift{\cdot} y \subseteq z$. 
\end{definition}

\begin{theorem}[\cite{GGP} Theorem 6.1]\label{GGPThm}
The extended Stone dual $(X,R)$ of $(\pow_{\up}(\bb{N}),\lift{+})$\footnote{In \cite{GGP} the Boolean algebra of regular languages over a fixed alphabet $A$ is considered together with the two residuals of the concatenation of languages. We are only considering the case where $A$ is a singleton, and in this case a language can be identified with the set of lengths of words in that language. Moreover the operation of concatenation as well as the two residuals of concatenation are all easily seen to be definable from one another.} is functional, and $(X,\times)$ is the free profinite monoid on one generator. 
\end{theorem}

Now that the context has been set up, we will establish the connection between the free profinite monoid on one generator and the pseudofinite monadic second order theory of linear order that we have analysed in \cref{SecAxioms,SecCompletions}.  

\begin{definition}[Definition of Lindenbaum-Tarski algebra and space of $0$-types.]
Let $\lang{}$ be a first order language, $T$ an $\lang{}$-theory. 

The \textbf{Lindenbaum-Tarski algebra} $\LT(T)$ is the Boolean algebra comprising equivalence classes of $\lang{}$-sentences with respect to the equivalence relation $\phi \sim \psi$ given by,
\[
T \models \phi \leftrightarrow \psi,
\]
with the operations being the usual Boolean logical connectives. 

The \textbf{space of $0$-types} of $T$, $S_0(T)$, is the Boolean space with underlying set,
\[
\{T':T' \text{ complete and }T' \models T\},
\]
and basis of clopen sets,
\[
\dangle{\phi} \coloneqq \{T':T' \models T \cup \{\phi\} \},
\] 
one for each $\lang{}$-sentence $\phi$.
\end{definition}

\begin{remark}
It is well known that given a first order theory $T$, the Lindenbaum-Tarski algebra $\LT(T)$and the space of $0$-types $S_0(T)$ are Stone duals of one another. A complete theory extending $T$ and a prime filter on $\LT$ are virtually defined in the same way.
\end{remark}

\begin{notation}
Going forward, since we will only be concerned with $T_{\MSO(\Fin)}$ we will write $\LT$ as shorthand for $\LT(T_{\MSO(\Fin)})$ and $S_0$ as shorthand for $S_0(T_{\MSO(\Fin)})$.
\end{notation}

\begin{definition}[Definition of $+$ on the Lindenbaum-Tarski algebra of $T_{\MSO(\Fin)}$.]
Define an operation $+:\LT^2 \rightarrow \LT$ as follows, given $\phi,\psi \in \LT$ take $\phi + \psi$ to be the (equivalence class of the) sentence,
\[
\eta: \exists X (\dwcl(X) \wedge \phi^X \wedge \psi^{X^c}).
\]
See \cref{ResRelRemark} for unpacking of the notation used, in particular the definition of the formula $\dwcl(X)$.
\end{definition}

\begin{proposition}\label{PropFSPECIso}
The map $\FSpec:(\LT,+) \rightarrow (\pow_{\up}(\bb{N}),\lift{+})$,  which we defined in \cref{FSpecDef}, is an isomorphism in the category of Boolean algebras equipped with a normal additive binary operation.
\end{proposition}
\begin{proof}
Recall that in \Cref{SPECUPThm} we said that $\FSpec(\phi)$ is ultimately periodic for each $\phi \in \LT$, so we have a well defined function. It is straightforward to show that this map preserves joins, meets, and complements. The map is injective by \cref{TMFinMainTheorem}, for if $n \in \FSpec(\phi) \wo \FSpec(\psi)$ then $\MSO(n) \models \phi \wedge \neg \psi$ and therefore $T_{\MSO(\Fin)} \not\models \phi \leftrightarrow \psi$. Surjectivity can be established directly, given $A \in \pow_{\up}(\bb{N})$ it is left as an exercise to produce $\phi_A \in \LT$ such that $\FSpec(\phi_A) = A$.

All that remains then is to show that for any $\phi,\psi \in \LT$ we have,
\[
\FSpec(\phi + \psi) = \FSpec(\phi) \lift{+} \FSpec(\psi).
\]
Let $\eta = \phi + \psi$. Now take $n \in \FSpec(\phi + \psi)$, i.e. suppose $\MSO(n) \models \eta$. This means that,
\[
\MSO(n) \models \exists X (\dwcl(X) \wedge \phi^X \wedge \psi^{X^c}),
\]
say $\MSO(n) \models \dwcl(A) \wedge \phi^A \wedge \psi^{A^c}$. Then clearly $A$ is of the form $\{0,\ldots,i-1\}$ for some $i \leq n$. We therefore get $i \in \FSpec(\phi)$ and $n-i \in \FSpec(\psi)$, so $n \in \FSpec(\phi) \lift{+} \FSpec(\psi)$.

Conversely, let $i \in \FSpec(\phi)$ and $j \in \FSpec(\psi)$ be given. Then it is immediate that $\MSO(i+j) \models \eta$, as taking $A = \{0,\ldots,i\}$ we have,
\[
\MSO(i+j) \models \phi^A \wedge \psi^{A^c}.
\]
So $\FSpec(\phi + \psi) = \FSpec(\phi) \lift{+} \FSpec(\psi)$ as required.
\end{proof}

The following now falls out of \cref{GGPThm,PropFSPECIso}.

\begin{corollary}\label{CorLTdualFPM}
The extended Stone dual of $(\LT,+)$ is the free profinite monoid on one generator. 
\end{corollary}

\begin{remark}
It follows immediately from \cref{CorLTdualFPM} that the extended Stone dual of $(\LT,+)$ is functional (this is an essential part of what is proved by Gehrke, Grigorieff, and Pin to establish \cref{GGPThm}). We include a direct proof of this in the proof of the following statement (the proof that $x \msotimes y \subseteq \dangle{x \lift{+} y}$) as it is illustrative of how the connection with model-theory can allow for a different perspective on the structures involved. 
\end{remark}

\begin{theorem}\label{ThmS0dualLT}
$(S_0,\msotimes)$ is the extended Stone dual of $(\LT,+)$. 
\end{theorem}
\begin{proof}
By definition the extended Stone dual of $(\LT,+)$ is $(S_0,R_+)$ where $R_+$ is the ternary relation on $S_0$ is given by,
\[
R_+(x,y,z) \Leftrightarrow x \lift{+} y \subseteq z.
\]
We will show that $\dangle{x \lift{+} y}\footnote{Recall that if $L$ is a lattice and $S \subseteq L$ is a subset then $\dangle{S}$ is the filter generated by $S$, i.e. $\dangle{S} = \bigcup_{n \in \bb{N}} \{a \in L: a \geq \bigwedge_{i=1}^n s_i$ for some $s_1,\ldots,s_n \in S\}$.} = x \msotimes y$. Recall that over a Boolean algebra prime filters are maximal filters (i.e. ultrafilters) and therefore we can never have a proper inclusion of prime filters of a Boolean algebra. Hence from $x \lift{+} y = x \msotimes y$ we can deduce that $x \msotimes y$ is the unique prime filter containing $x \lift{+} y$. 

Firstly we show $x \lift{+} y \subseteq x \msotimes y$. Let $\phi \in x$ and $\psi \in y$ be given. By definition $\phi + \psi$ is the sentence $\exists X (\dwcl(X) \wedge \phi^X \wedge \psi^{X^c}) \in \LT$. But now recall that $x \msotimes y$ is the theory of $M \msotimes N$ for any $M \models x$ and $N \models y$ (see \cref{TMFinSumCor}). Then consider the element $A = (\top,\bot) \in M \msotimes N$. This element $A$ is clearly downwards closed, $M \msotimes N \upharpoonright A \models \phi$, and $M \msotimes N \upharpoonright A^c \models \psi$. Therefore $M \msotimes N \models \phi + \psi$, and hence $\phi + \psi \in x \msotimes y$.

To finish off we show that $x \msotimes y \subseteq \dangle{x \lift{+} y}$. We will in fact show that each sentence $\eta \in x \msotimes y$ is equivalent over $T_{\MSO(\Fin)}$ to a finite disjunction of sentences contained in $x \lift{+} y$. For this we will use the Fra\"isse-Hintikka theorem (see \cite{HodgesBible} Theorem 3.3.2) in a similar way as we did in \cref{TMFinEFStandardThm}. Take $\eta \in x \msotimes y$, this means there are $M \models x$ and $N \models y$ such that $M \msotimes N \models \eta$. We want to show that there exist $m \in \bb{N}$, $\phi_1,\ldots,\phi_m$, and $\psi_1,\ldots,\psi_m$ such that for $i=1,\ldots,m$,
\[
T_{\MSO(\Fin)} \models \eta \leftrightarrow \bigvee_{i=1}^m (\phi_i + \psi_i).
\] 
Without loss of generality\footnote{See \cite{HodgesBible} Corollary 2.6.1 pp 59.} we can assume that $\eta$ is unnested of quantifier rank $k \in \bb{N}$. The equivalence relation $\approx_k$ has finitely many equivalence classes, each of which is axiomatised by an $\lang{\msotimes}$-sentence. By \cref{SumEFThm}, $\msotimes$ is a well defined operation on the set of equivalence classes for $\approx_k$. Now take the finite disjunction of all sentences of the form $\phi + \psi$ where $\phi,\psi$ each axiomatise equivalence classes of $\approx_k$ which $\msotimes$ sends to an equivalence class of $\approx_k$ containing a model of $\eta$. This disjunction is equivalent to $\eta$ over $T_{\MSO(\Fin)}$ so we are done. 
\end{proof}

\begin{theorem}\label{ThmS0isFPM}
$(S_0,\msotimes)$ is the free profinite monoid on one generator.
\end{theorem}
\begin{proof}
Combine \cref{CorLTdualFPM,ThmS0dualLT}.
\end{proof}

Now the reader is invited to think about \cref{ThmS0isFPM} and \cref{RmkResidueAddition} together. 
\clearpage
\printbibliography[heading=bibintoc]

@InCollection{SEPHOL,
	author       =	{Väänänen, Jouko},
	title        =	{Second-order and Higher-order Logic},
	booktitle    =	{The Stanford Encyclopedia of Philosophy},
	editor       =	{Edward N. Zalta},
	howpublished =	{\url{https://plato.stanford.edu/archives/fall2020/entries/logic-higher-order/}},
	year         =	{2020},
	edition      =	{Fall 2020},
	publisher    =	{Metaphysics Research Lab, Stanford University}
}

@incollection {BSAxOmega1,
    AUTHOR = {B\"{u}chi, J. Richard and Siefkes, Dirk},
     TITLE = {Axiomatization of the monadic second order theory of {$\omega
              _{1}$}},
 BOOKTITLE = {The monadic second order theory of all countable ordinals
              ({D}ecidable theories, {II})},
     PAGES = {129--217. Lecture Notes in Math., Vol. 328},
      YEAR = {1973},
   MRCLASS = {02G10 (06A05)},
  MRNUMBER = {0479997},
}

@mastersthesis{Dekkers,
  title={Stone duality: An application in the theory of formal languages},
  author={Dekkers, Mirte},
  school={Radboud University Nijmegen},
  year={2008}
}

@Inbook{AlmeidaProfinBook,
author="Almeida, Jorge
and Costa, Alfredo
and Kyriakoglou, Revekka
and Perrin, Dominique",
title="Free Profinite Monoids, Semigroups and Groups",
bookTitle="Profinite Semigroups and Symbolic Dynamics",
year="2020",
publisher="Springer International Publishing",
address="Cham",
pages="87--137",
isbn="978-3-030-55215-2",
doi="10.1007/978-3-030-55215-2_4",
url="https://doi.org/10.1007/978-3-030-55215-2_4"
}

@InProceedings{GGP,
author="Gehrke, Mai
and Grigorieff, Serge
and Pin, Jean-{\'E}ric",
editor="Abramsky, Samson
and Gavoille, Cyril
and Kirchner, Claude
and Meyer auf der Heide, Friedhelm
and Spirakis, Paul G.",
title="A Topological Approach to Recognition",
booktitle="Automata, Languages and Programming",
year="2010",
publisher="Springer Berlin Heidelberg",
address="Berlin, Heidelberg",
pages="151--162",
isbn="978-3-642-14162-1"
}

@article {Hansoul,
    AUTHOR = {Hansoul, G.},
     TITLE = {A duality for {B}oolean algebras with operators},
   JOURNAL = {Algebra Universalis},
  FJOURNAL = {Algebra Universalis},
    VOLUME = {17},
      YEAR = {1983},
    NUMBER = {1},
     PAGES = {34--49},
      ISSN = {0002-5240},
   MRCLASS = {06D05 (06E15)},
  MRNUMBER = {709996},
MRREVIEWER = {\v{Z}arko Mijajlovi\'{c}},
       DOI = {10.1007/BF01194512},
       URL = {https://doi.org/10.1007/BF01194512},
}

@article {Goldblatt,
    AUTHOR = {Goldblatt, Robert},
     TITLE = {Varieties of complex algebras},
   JOURNAL = {Ann. Pure Appl. Logic},
  FJOURNAL = {Annals of Pure and Applied Logic},
    VOLUME = {44},
      YEAR = {1989},
    NUMBER = {3},
     PAGES = {173--242},
      ISSN = {0168-0072},
   MRCLASS = {08B99 (03C05)},
  MRNUMBER = {1020344},
       DOI = {10.1016/0168-0072(89)90032-8},
       URL = {https://doi.org/10.1016/0168-0072(89)90032-8},
}

@article {JT1,
    AUTHOR = {J\'{o}nsson, Bjarni and Tarski, Alfred},
     TITLE = {Boolean algebras with operators. {I}},
   JOURNAL = {Amer. J. Math.},
  FJOURNAL = {American Journal of Mathematics},
    VOLUME = {73},
      YEAR = {1951},
     PAGES = {891--939},
      ISSN = {0002-9327},
   MRCLASS = {09.0X},
  MRNUMBER = {44502},
MRREVIEWER = {R. C. Lyndon},
       DOI = {10.2307/2372123},
       URL = {https://doi.org/10.2307/2372123},
}

@article {JT2,
    AUTHOR = {J\'{o}nsson, Bjarni and Tarski, Alfred},
     TITLE = {Boolean algebras with operators. {II}},
   JOURNAL = {Amer. J. Math.},
  FJOURNAL = {American Journal of Mathematics},
    VOLUME = {74},
      YEAR = {1952},
     PAGES = {127--162},
      ISSN = {0002-9327},
   MRCLASS = {09.1X},
  MRNUMBER = {45086},
MRREVIEWER = {R. C. Lyndon},
       DOI = {10.2307/2372074},
       URL = {https://doi.org/10.2307/2372074},
}

@article {GehrkeJonsson,
    AUTHOR = {Gehrke, Mai and J\'{o}nsson, Bjarni},
     TITLE = {Bounded distributive lattices with operators},
   JOURNAL = {Math. Japon.},
  FJOURNAL = {Mathematica Japonica},
    VOLUME = {40},
      YEAR = {1994},
    NUMBER = {2},
     PAGES = {207--215},
      ISSN = {0025-5513},
   MRCLASS = {06D05 (54F05)},
  MRNUMBER = {1297234},
MRREVIEWER = {J. Kar\'{a}sek},
}

@article {BSCountableOrd,
    AUTHOR = {B\"{u}chi, J. Richard and Siefkes, Dirk},
     TITLE = {The complete extensions of the monadic second order theory of
              countable ordinals},
   JOURNAL = {Z. Math. Logik Grundlag. Math.},
  FJOURNAL = {Zeitschrift f\"{u}r Mathematische Logik und Grundlagen der
              Mathematik},
    VOLUME = {29},
      YEAR = {1983},
    NUMBER = {4},
     PAGES = {289--312},
      ISSN = {0044-3050},
   MRCLASS = {03C85 (03B15 03C65 03E10)},
  MRNUMBER = {715253},
MRREVIEWER = {H.-D. Ebbinghaus},
       DOI = {10.1002/malq.19830290502},
       URL = {https://doi.org/10.1002/malq.19830290502},
}

@article {ShelahMTO,
    AUTHOR = {Shelah, Saharon},
     TITLE = {The monadic theory of order},
   JOURNAL = {Ann. of Math. (2)},
  FJOURNAL = {Annals of Mathematics. Second Series},
    VOLUME = {102},
      YEAR = {1975},
    NUMBER = {3},
     PAGES = {379--419},
}

@article {vGSteinberg,
    AUTHOR = {van Gool, Samuel J. and Steinberg, Benjamin},
     TITLE = {Pro-aperiodic monoids via saturated models},
   JOURNAL = {Israel J. Math.},
  FJOURNAL = {Israel Journal of Mathematics},
    VOLUME = {234},
      YEAR = {2019},
    NUMBER = {1},
     PAGES = {451--498},
      ISSN = {0021-2172},
   MRCLASS = {20M07 (03C98 08A50 20M35)},
  MRNUMBER = {4040833},
MRREVIEWER = {Jorge Almeida},
       DOI = {10.1007/s11856-019-1947-6},
       URL = {https://doi.org/10.1007/s11856-019-1947-6},
}

@book {HodgesBible,
    AUTHOR = {Hodges, Wilfrid},
     TITLE = {Model theory},
    SERIES = {Encyclopedia of Mathematics and its Applications},
    VOLUME = {42},
 PUBLISHER = {Cambridge University Press, Cambridge},
      YEAR = {1993},
     PAGES = {xiv+772},
      ISBN = {0-521-30442-3},
   MRCLASS = {03-01 (03-02 03Cxx)},
  MRNUMBER = {1221741},
MRREVIEWER = {J. M. Plotkin},
       DOI = {10.1017/CBO9780511551574},
       URL = {https://doi.org/10.1017/CBO9780511551574},
}

@phdthesis{Doets,
  author       = {Kees Doets}, 
  title        = {Completeness and definability : applications of the Ehrenfeucht game in second-order and intensional logic},
  school       = {Universiteit van Amsterdam},
  year         = {1987}
}

@incollection{Dauben,
   booktitle = {The Mathematics of Egypt, Mesopotamia, China, India, and Islam: A Sourcebook},
   title = {Chinese Mathematics},
   editor = {Victor J. Katz (ed.)},
   author = {Joseph W. Dauben},
   publisher = {Princeton University Press},
   year = {2007},
   pages = {187-384},
   chapter = {3}
}

@book{KNAutomataBook,
  doi = {10.1007/978-1-4612-0171-7},
  url = {https://doi.org/10.1007/978-1-4612-0171-7},
  year = {2001},
  publisher = {Birkh\"{a}user Boston},
  author = {Bakhadyr Khoussainov and Anil Nerode},
  title = {Automata Theory and its Applications}
}

\textit{E-mail address}: \texttt{deacon.linkhorn@manchester.ac.uk}\\
\textit{Personal website}: \url{http://www.deaconlinkhorn.com}
\end{document}